\documentclass[10pt]{article}

\pdfoutput=1 %ArXiv PDFLATEX
\usepackage{etex}
\usepackage[english]{babel}
\usepackage[utf8]{inputenc}
\usepackage[T1]{fontenc}
\usepackage{lmodern}
\usepackage{authblk}

\usepackage{amsmath, latexsym, amssymb, amsfonts, amsthm}
\usepackage[colorlinks,citecolor=black,filecolor=black,linkcolor=black,urlcolor=black ]{hyperref}
\usepackage{graphicx}
\usepackage{subfigure} 
\usepackage{enumitem}
\usepackage{multicol}

\def\be{\begin{equation}}   \def\ee{\end{equation}}
\def\ba   {\begin{array}}      \def\ea   {\end{array}}
\def\bea  {\begin{eqnarray}}   \def\eea  {\end{eqnarray}}
\def\bean {\begin{eqnarray*}}  \def\eean {\end{eqnarray*}}

\newtheorem{theorem} {Theorem}
\newtheorem{lemma}{Lemma}
\newtheorem{definition} {Definition}

\newtheorem{corollary} {Corollary}
\newtheorem{remark}{Remark}

\theoremstyle{definition}

\newcommand{\mN}{\ensuremath{\mathbb{N}}}

\newcommand{\mR}{\ensuremath{\mathbb{R}}}
\newcommand{\mT}{\ensuremath{\mathbb{T}}}

\newcommand{\cM}{\ensuremath{\mathcal{M}}}
\newcommand{\mM}{\ensuremath{\mathbb{M}}}

\newcommand{\cT}{\ensuremath{\mathcal{T}}}

\newcommand{\vspan}{\ensuremath{\mathrm{span}}}

\newcommand{\gm}[1]{\ensuremath{\gamma_{#1}}}
\newcommand{\gmc}[1]{\ensuremath{\overline{\gamma}_{#1}}}
\newcommand{\im}[1]{\ensuremath{\mathrm{\mathbf{i}}_{\mathbf{#1}}}}
\newcommand{\ibasic}[1]{\ensuremath{i_{#1}}}

\newcommand{\mI}{\ensuremath{\mathbb{I}}}
\newcommand{\mS}{\ensuremath{\mathbb{S}}}
\newcommand{\iter}{\ensuremath{\mathrm{It}}}
\usepackage{todonotes}

\title{Characterization of the Principal 3D Slices Related to the Multicomplex Mandelbrot Set}

\author{Guillaume Brouillette\thanks{E-mail: {\tt guillaume.brouillette@uqtr.ca}} }
\author{Dominic Rochon\thanks{E-mail: {\tt dominic.rochon@uqtr.ca}}}
\affil{Département de mathématiques et d'informatique, Université du Québec\\
 C.P. 500, Trois-Rivières, Québec, Canada, G9A 5H7.}
\date{\today}

\begin{document}
\maketitle

%%Abstract
\begin{abstract}
This article focuses on the dynamics of the different tridimensional principal slices of the multicomplex Multibrot sets. First, we define an equivalence relation between those slices. Then, we characterize them in order to establish similarities between their behaviors. Finally, we see that any multicomplex tridimensional principal slice is equivalent to a tricomplex slice up to an affine transformation. This implies that, in the context of tridimensional principal slices, Multibrot sets do not need to be generalized beyond the tricomplex space.
\end{abstract}\vspace{0.5cm} 

\section*{Introduction}

The multicomplex space is one of the multiple generalizations of the complex space \cite{GarantPelletier, Baley, RochonMaitrise}. The set of multicomplex numbers is defined, essentially, by introducing more imaginary units and multicomplex addition and multiplication are analogous to the complex operations. Thus, working with multicomplex numbers is rather intuitive. Moreover, what makes this generalization interesting is that many results and concepts known in the complex space can be extended \cite{GarantRochon, RochonMaitrise, Rochon3, Rochon2, Holomorphy, vajiac}.

Here, we are mostly interested in multicomplex fractals. Indeed, some fractals originally defined in the complex plane can be generalized to the multicomplex space \cite{Martineau, RochonMartineau, Parise, Rochon1, Rochon2, Wang}. In particular, we are looking into multicomplex Multibrot sets, which are a generalization of the Mandelbrot set.

Furthermore, it will be seen later on that those objects have over three real dimensions, meaning that they do not have a graphical representation. Therefore, we can only partially visualize them by defining principal 3D slices. Since those fractals, and consequently their slices, are defined dynamically using polynomial iterations, it is essential to study multicomplex dynamics in order to understand them. The iteration of complex polynomials has been studied for many years \cite{Beardon, Douady, Milnor2}. However, the iteration of multicomplex polynomials has only been explored much more recently \cite{ RochonRansfordParise, RochonParise, RochonParise2}.

When slices have the same dynamics, they also look the same. This leads to defining an equivalence relation between them. This article establishes that, when it comes to principal 3D slices, any multicomplex principal 3D slice of a Multibrot set is equivalent to at least one quadricomplex slice or directly to one tricomplex slice up to an affine transformation. In other words, in that context, it is not necessary to explore principal 3D slices beyond the tricomplex space. Hence, the tricomplex space is, in a way, optimal.

The article goes as follows. In Section 1, multicomplex numbers are introduced. Then, in Section 2, we define multicomplex Multibrot sets and their principal 3D slices. In addition, we define a relation between those slices and prove that it is an equivalence relation. In Section 3, the dynamics of those slices are seen thoroughly. Finally, in Section 4, we show why the tricomplex space is optimal in this context.

\section{Multicomplex Numbers}

\subsection{Basic concepts}

We present here a short summary of the concepts on multicomplex numbers preliminary to the main results.

It is well known that a complex number is defined using two real components and an imaginary unit $\ibasic{1}$ such that $\ibasic{1}^2=-1$. Multicomplex numbers of order $n$, also called $n$-complex numbers, are obtained by using this idea recursively. Indeed, for any integer $n\geq 1$, the set of multicomplex numbers of order $n$ is defined as
\[\mM(n) := \{\eta_1 + \eta_2\ibasic{n} : \eta_1,\eta_2\in\mM(n-1)\}\]
with $\ibasic{n}^2 = -1$ and $\mM(0):=\mR$ \cite{GarantPelletier,GarantRochon}. Moreover, multicomplex addition and multiplication are defined similarly to the analogous complex operations, meaning that
\begin{align*}
&(\eta_1 + \eta_2\ibasic{n}) + (\zeta_1 + \zeta_2\ibasic{n}) = (\eta_1 + \zeta_1) + (\eta_2 + \zeta_2)\ibasic{n};\\
&(\eta_1 + \eta_2\ibasic{n})(\zeta_1 + \zeta_2\ibasic{n}) = (\eta_1\zeta_1 - \eta_2\zeta_2) + (\eta_1\zeta_2 + \eta_2\zeta_1)\ibasic{n}.
\end{align*}

Using these basic operations, we can see that any $n$-complex number may be expanded to $2^n$ terms with real coefficients \cite{GarantPelletier, Baley}. Each term then corresponds to a combination of imaginary units. For example, a bicomplex number $\eta$ may be expressed as
\begin{align*}
\eta = \eta_1 + \eta_2\ibasic{2} = x_1 + x_2\ibasic{1} + x_3\ibasic{2} + x_4\ibasic{1}\ibasic{2}
\end{align*}
assuming that $\eta_1 = x_1 + x_2\ibasic{1}$ and $\eta_2 = x_3 + x_4\ibasic{1}$. Let $\mI(n)$ be the set containing the unit 1 and all combinations of $\{\ibasic{1},\ibasic{2},...,\ibasic{n}\}$. For instance, $\mI(1) = \{ 1,\, \ibasic{1} \}$, $\mI(2) = \{ 1,\, \ibasic{1},\, \ibasic{2},\, \ibasic{1}\ibasic{2} \}$ and $\mI(3) = \{ 1,\, \ibasic{1},\, \ibasic{2},\, \ibasic{3},\, \ibasic{1}\ibasic{2},\, \ibasic{1}\ibasic{3},\, \ibasic{2}\ibasic{3},\, \ibasic{1}\ibasic{2}\ibasic{3} \}$. Generally, for all $\eta\in\mM(n)$, we have that
\[\eta = \sum_{\im{}\in\mI(n)}x_{\im{}}\im{}\]
where $x_{\im{}}\in\mR$ and $|\mI(n)| = 2^n$. Therefore, all numbers $\eta\in\mM(n)$ cannot be represented graphically when $n\geq 2$.

Notice that some units $\im{}\in\mI(n)$ are such that $\im{}^2 = 1$ with $\im{}\neq 1$. Those are called \textbf{hyperbolic} \cite{RochonShapiro, Sobczyk}. Using combinatorics, it could be proven that $\mI(n)$ contains $2^{n-1}$ complex imaginary units and $2^{n-1}-1$ hyperbolic units.

It can easily be verified that $(\mM(n),+,\cdot)$ is a commutative unitary ring. Moreover, the set $\mM(n)$ together with multicomplex addition and multiplication by real numbers is a vector space over the field $\mR$ and can be viewed as a direct sum of complex spaces. We can also define the norm $\lVert\cdot\rVert_n$ of a $n$-complex number as the Euclidean norm of its representation in $\mR^{2^n}$.

\subsection{Idempotent representation}\label{secIdempotent}

The last subsection gives a good general idea of what a multicomplex number is. Nonetheless, one more concept will be necessary later on. In the $n$-complex space, when $n\geq 2$, there exists idempotent numbers $\gamma$, meaning that $\gamma^2 = \gamma$. In particular, for $1\leq h < n$, consider
\[\gm{h} = \frac{1 + \ibasic{h}\ibasic{h+1}}{2} \quad\text{and}\quad \gmc{h} = \frac{1 - \ibasic{h}\ibasic{h+1}}{2}.\]
In addition, these two numbers are orthogonal, meaning that $\gm{h}\gmc{h} = 0$.

Given an arbitrary $n$-complex number $\eta = \eta_1 + \eta_2\ibasic{n}$, we can see that
\[\eta = (\eta_1 - \eta_2\ibasic{n-1})\gm{n-1} + (\eta_1 + \eta_2\ibasic{n-1})\gmc{n-1}.\]
This is called the \textbf{idempotent representation} of $\eta$. Using the properties of $\gm{n-1}$ and $\gmc{n-1}$, we can see that
\begin{enumerate}
\item $(\alpha_1\gm{n-1} + \alpha_2\gmc{n-1}) + (\beta_1\gm{n-1} + \beta_2\gmc{n-1}) = (\alpha_1 + \beta_1)\gm{n-1} + (\alpha_2 + \beta_2)\gmc{n-1}$;
\item $(\alpha_1\gm{n-1} + \alpha_2\gmc{n-1}) \cdot (\beta_1\gm{n-1} + \beta_2\gmc{n-1}) = (\alpha_1\beta_1)\gm{n-1} + (\alpha_2\beta_2)\gmc{n-1}$;
\item $(\alpha_1\gm{n-1} + \alpha_2\gmc{n-1})^p = \alpha_1^p\gm{n-1} + \alpha_2^p\gmc{n-1}\mbox{ } \forall p\in\mathbb{N}.$
\end{enumerate}

The idempotent components of any $\eta\in\mM(n)$ can also be written using the idempotent representation. In fact, we can expand $\eta$ until it is expressed with $2^{n-1}$ idempotent components. More explicitly, consider $S_h$ such that
\begin{align*}
S_h = \begin{cases}
\{\gm{1},\gmc{1}\} & \text{when } h = 1,\\
\gm{h}\cdot S_{h-1}\bigcup \gmc{h}\cdot S_{h-1} &\text{otherwise}.
\end{cases}
\end{align*}
Then, there exists $2^{n-1}$ numbers $\eta_{\gamma}$ such that
\[\eta = \sum_{\gamma\in S_{n-1}}\eta_{\gamma}\gamma.\]
Under this form, all components $\eta_{\gamma}$ are complex, namely $\eta_{\gamma}\in\mM(1)$.

Furthermore, an operation similar to the Cartesian product may be defined. Indeed, consider two sets $A,B\subseteq \mM(n-1)$. Then, we define the product $\times_{\gm{n-1}}$ as
\[A\times_{\gm{n-1}}B := \left\lbrace \eta\gm{n-1} + \zeta\gmc{n-1}\ |\ (\eta,\zeta)\in A\times B\right\rbrace.\]
We will see in the next sections that properties and results in $\mM(n-1)$ can be extended to $\mM(n)$ using this product.

\section{Generalized Mandelbrot Sets}

We present here an intuitive generalization of the complex Mandelbrot set to the multicomplex Multibrots sets.

Let $Q_{p,c}(\eta) = \eta^p + c$ and denote
\[Q^m_{p,c}(\eta) = \underbrace{\left(Q_{p,c}\circ Q_{p,c}\circ\dots\circ Q_{p,c}\right)}_{m\text{ times}}(\eta).\]
Using the function $Q_{p,c}$, we can define the classical Mandelbrot set as
\[\cM = \big\{ c\in\mM(1)\ :\ \{Q^m_{2,c}(0)\}_{m=1}^\infty\text{ is bounded } \big\}.\]
We can easily modify this last definition to obtain the following more general one.

\begin{definition}
Let $n,p\in\mN$ such that $p\geq 2$. The $n$-complex \textbf{Multibrot} set of order $p$ is defined as
\[\cM_n^p = \big\{ c\in\mM(n)\ :\ \{Q^m_{p,c}(0)\}_{m=1}^\infty\text{ is bounded } \big\}.\]
\end{definition}

Complex Multibrot sets are seen in many references (see \cite{Baribeau, Parise, RochonRansfordParise, RochonParise, RochonParise2} for example). Although they have been generalized up to the tricomplex space in some of those articles, their generalization to $n$-complex space has not often been seen in the literature. The specific case of $\cM_3^2$ is called the \textit{Metatronbrot}. Here is an interesting property of the multicomplex Multibrot sets based on the idempotent representation.

\begin{theorem}\label{theoMultibrotIdem}
Let $n,p\in\mN$ such that $n,p\geq 2$. We have that
\[\cM_n^p = \cM_{n-1}^p\times_{\gm{n-1}}\cM_{n-1}^p.\]
\end{theorem}

\begin{proof}
This theorem is a generalization of results presented in \cite{GarantPelletier, GarantRochon, RochonParise}.

Essentially, the result follows from the properties of the idempotent representation. Let $c = c_{\gm{n-1}}\gm{n-1} + c_{\gmc{n-1}}\gmc{n-1}$. Then, using the induction principle on $m$, we can show that
\[Q^m_{p,c}(0) = Q^m_{p,c_{\gm{n-1}}}(0)\gm{n-1} + Q^m_{p,c_{\gmc{n-1}}}(0)\gmc{n-1}.\]
Furthermore, from \cite{Baley}, we know that
\[\lVert \eta \rVert_n = \sqrt{\frac{\lVert \eta_{\gm{n-1}} \rVert_{n-1}^2 + \lVert \eta_{\gmc{n-1}} \rVert_{n-1}^2}{2}}\quad \forall \eta = \eta_{\gm{n-1}}\gm{n-1} + \eta_{\gmc{n-1}}\gmc{n-1}\]
where $\lVert\cdot\rVert_n$ is the Euclidean norm of a $n$-complex number. Thus, we know that $\left\lbrace Q^m_{p,c}(0)\right\rbrace_{m=1}^\infty$ remains bounded if and only if both $\left\lbrace Q^m_{p,c_{\gm{n-1}}}(0)\right\rbrace_{m=1}^\infty$ and $\left\lbrace Q^m_{p,c_{\gmc{n-1}}}(0)\right\rbrace_{m=1}^\infty$ do as well. Hence, 
\begin{align*}
c\in\cM_n^p &\Leftrightarrow (c_{\gm{n-1}},c_{\gmc{n-1}})\in\cM_{n-1}^p\times\cM_{n-1}^p,\\
&\Leftrightarrow c_{\gm{n-1}}\gm{n-1} + c_{\gmc{n-1}}\gmc{n-1}\in\cM_{n-1}^p\times_{\gm{n-1}}\cM_{n-1}^p.\qedhere
\end{align*}
\end{proof}

\begin{corollary}\label{coroMultibrotIdem}
Let $n,p\in\mN$ such that $n\geq 2$ and $p\geq 2$. Consider $c\in\mM(n)$ such that
\[c = \sum_{\gamma\in S_{n-1}}c_{\gamma}\gamma.\]
We have that
\[c\in\cM_n^p\Leftrightarrow c_{\gamma}\in\cM^p\ \forall \gamma\in S_{n-1}.\]
\end{corollary}

\begin{proof}
The proof is done using the induction principle. When $n=2$, we see the proposition is true using Theorem \ref{theoMultibrotIdem}. Then, assuming the proposition is true for some value $n-1\geq 2$, it follows from Theorem \ref{theoMultibrotIdem} that
\begin{align*}
c\in\cM_n^p &\Leftrightarrow c_{\gm{n-1}},c_{\gmc{n-1}}\in\cM_{n-1}^p \\
&\Leftrightarrow c_{\gamma}\in\cM^p\ \forall \gamma\in S_{n-1}.\qedhere
\end{align*}
\end{proof}

As $\cM_n^p$ is a subset of a $2^n$-dimensional space, it cannot be represented in a graph when $n\geq 2$. Therefore, the Multibrot sets can only be partially visualized by extracting 3D slices. The next definitions are generalizations of definitions in \cite{GarantPelletier, GarantRochon, Parise, RochonParise, RochonParise2}.
	
\begin{definition}
Let $\im{m},\im{k},\im{l}\in\mI(n)$ with $\im{m} \neq \im{k}$, $\im{m} \neq \im{l}$ and $\im{k} \neq \im{l}$. We define the following vector subspace of $\mM(n)$: 
\begin{align*}
\mT(\im{m},\im{k},\im{l}) &:= \vspan_{\mR} \{\im{m},\im{k},\im{l}\}.
\end{align*}
\end{definition}

\begin{remark}
It is important to remember that $\im{m},\im{k},\im{l}\in\mI(n)$ are not necessarily complex imaginary units. They could be real, complex or hyperbolic units.
\end{remark}

\begin{remark}
The notation $\vspan_{\mR}$ stands for the linear space spanned by some vectors over the field $\mR$. Equivalently, it stands for the space of all finite linear combinations of those vectors.
\end{remark}
	
\begin{definition}
Let $\im{m},\im{k},\im{l}\in\mI(n)$ with $\im{m} \neq \im{k}$, $\im{m} \neq \im{l}$ and $\im{k} \neq \im{l}$. We define a principal 3D slice of the Multibrot set $\cM_n^p$ as
\begin{align*}
\cT^p (\im{m} , \im{k} , \im{l} ) = \left\lbrace c \in \mT(\im{m},\im{k},\im{l}) : \left\lbrace Q_{p,c}^m (0) \right\rbrace_{m=1}^{\infty} \text{ is bounded} \right\rbrace.
\end{align*}
\end{definition}

\begin{remark}
When the context is clear, we write $\cT^p$ instead of $\cT^p(\im{m}, \im{k}, \im{l})$.
\end{remark}

A relation between the principal 3D slices of the Multibrot sets may be defined. An important subspace must be presented beforehand.

\begin{definition}
Let $\im{m},\im{k},\im{l}\in\mI(n)$ with $\im{m} \neq \im{k}$, $\im{m} \neq \im{l}$ and $\im{k} \neq \im{l}$. Then,
\begin{align*}
\iter^p(\im{m},\im{k},\im{l}) &:= \vspan_{\mR}\{Q_{p,c}^m(0) : c\in\mT(\im{m},\im{k},\im{l})\text{ and } m\in\mN\}.
\end{align*}
\end{definition}

In other words, the subspace $\iter^p(\im{m},\im{k},\im{l})$ is the smallest subspace of $\mM(n)$ containing all iterates $Q_{p,c}^m(0)$ with $c\in\mT(\im{m},\im{k},\im{l})$. This concept is used to define the following relation between principal 3D slices. In \cite{Parise}, the next definition is presented specifically for the tricomplex case.

\begin{definition}\label{defEquivRel}
Let $\cT_1^p(\im{m},\im{k},\im{l})$ and $\cT_2^p(\im{r},\im{q},\im{s})$ be two principal 3D slices of $\cM_n^p$. Consider the sets $M_1=\iter^p(\im{m},\im{k},\im{l})$ and $M_2=\iter^p(\im{r},\im{q},\im{s})$. We say that $\cT_1^p \sim \cT_2^p$ if there exists a linear bijective application $\varphi : M_1 \rightarrow M_2$ such that $\varphi\left(\mT(\im{m},\im{k},\im{l})\right)=\mT(\im{r},\im{q},\im{s})$ and, for all $c\in\mT(\im{m},\im{k},\im{l})$,
\begin{align*}
(\varphi \circ Q_{p,c} \circ \varphi^{-1} ) (\eta ) = Q_{p, \varphi(c)} (\eta)\ \forall \eta \in M_2.
\end{align*}
In this case, we say that $\cT_1^p$ and $\cT_2^p$ have the same dynamics.
\end{definition}

\begin{remark}
To lighten the text, whenever we consider units $\im{m},\im{k},\im{l}\in\mI(n)$, we assume that $\im{m} \neq \im{k}$, $\im{m} \neq \im{l}$ and $\im{k} \neq \im{l}$. Analogously, we always assume $\im{r} \neq \im{q}$, $\im{r} \neq \im{s}$ and $\im{q} \neq \im{s}$ when considering units $\im{r},\im{q},\im{s}\in\mI(n)$.
\end{remark}
	
In \cite{RochonParise}, the authors defined a similar but slightly different relation between the principal 3D slices of the Multibrot sets. We believe this definition is more accurate. Moreover, it is a generalization of Definition 3.8 presented in \cite{Parise} for the tricomplex case. To prove this statement, we will need a result from linear algebra. It is a consequence of the rank-nullity Theorem.

\begin{lemma}\label{lemRankNullity}(See \cite{LinearAlgebra}.)
Let $V,W$ be two vector spaces and $L:V\rightarrow W$ be a linear application. If $\dim(V) = \dim(W) < \infty$, then
\[L \text{ is injective } \Leftrightarrow L\text{ is surjective } \Leftrightarrow L \text{ is bijective.}\]
\end{lemma}

\begin{theorem}
When $n = 3$, Definition \ref{defEquivRel} is equivalent to Definition 3.8 in \cite{Parise}.
\end{theorem}

\begin{proof}
Consider Definition \ref{defEquivRel} assuming that $n=3$. We see that there is one main difference between both definitions:
\begin{enumerate}[label=(\arabic*)]
	\item in Definition \ref{defEquivRel}, we assume $\varphi$ is such that
	\[\varphi\left(\mT(\im{m},\im{k},\im{l})\right)=\mT(\im{r},\im{q},\im{s});\]
	\item in Definition 3.8 in \cite{Parise}, it is rather said that \[\forall c_2\in\mT(\im{r},\im{q},\im{s}),\, \exists c_1\in\mT(\im{m},\im{k},\im{l})\text{ such that }\varphi(c_1) = c_2.\]
\end{enumerate}
We can prove that $(1)\Leftrightarrow (2)$.

\begin{itemize}
	\item [$\Rightarrow)$] If $\varphi\left(\mT(\im{m},\im{k},\im{l})\right)=\mT(\im{r},\im{q},\im{s})$, then for all $c_2\in\mT(\im{r},\im{q},\im{s})$, we have directly that there exists a value $c_1\in\mT(\im{m},\im{k},\im{l})$ such that $\varphi(c_1) = c_2$.
	\item [$\Leftarrow)$] Now, suppose that for all $c_2\in\mT(\im{r},\im{q},\im{s})$, there exists $c_1\in\mT(\im{m},\im{k},\im{l})$ such that $\varphi(c_1) = c_2$. As $\varphi$ is bijective, we know that $\varphi^{-1}$ exists. The previous statement can therefore be written
	\[\varphi^{-1}(c_2) = c_1\in\mT(\im{m},\im{k},\im{l})\text{ for all } c_2\in\mT(\im{r},\im{q},\im{s}),\]
	meaning that $\varphi^{-1}\left(\mT(\im{r},\im{s},\im{q})\right)\subseteq\mT(\im{m},\im{k},\im{l})$. Moreover, we know that the restriction of the linear application $\varphi^{-1}$ to $\mT(\im{r},\im{s},\im{q})$ is injective. Then, Lemma \ref{lemRankNullity} implies that $\varphi^{-1}$ forms a bijection from $\mT(\im{r},\im{s},\im{q})$ to $\mT(\im{m},\im{k},\im{l})$. Hence, we conclude that $\varphi^{-1}\left(\mT(\im{r},\im{s},\im{q})\right)=\mT(\im{m},\im{k},\im{l})$ and, consequently, $\varphi\left(\mT(\im{m},\im{k},\im{l})\right)=\mT(\im{r},\im{s},\im{q})$.\qedhere	
\end{itemize}
\end{proof}

Furthermore, it has been proven in \cite{Parise} that this relation is an equivalence relation in the tricomplex space. The proof of this statement can be generalized to the multicomplex space.

\begin{theorem}
The relation $\sim$ from Definition \ref{defEquivRel} is an equivalence relation.
\end{theorem}

\begin{proof}
We need to prove that $\sim$ is reflexive, symmetric and transitive.
\begin{itemize}
	\item Let $\cT^p(\im{m},\im{k},\im{l})$ be a principal 3D slice and consider $M = \iter^p(\im{m},\im{k},\im{l})$. It is easy to see that $\cT^p\sim\cT^p$ by using the identity application $\varphi(\eta) = \eta$.
	
	\item Let $\cT^p_1(\im{m},\im{k},\im{l})$ and $\cT^p_2(\im{r},\im{q},\im{s})$ be two principal 3D slices and consider the sets $M_1 = \iter^p(\im{m},\im{k},\im{l})$ and $M_2 = \iter^p(\im{r},\im{q},\im{s})$. If $\cT^p_1\sim\cT^p_2$, then there exists a bijective linear application $\varphi:M_1\rightarrow M_2$ such that $\varphi\left(\mT(\im{m},\im{k},\im{l})\right)=\mT(\im{r},\im{q},\im{s})$ and, for all $c_1\in\mT(\im{m},\im{k},\im{l})$,
	\begin{align*}
	(\varphi \circ Q_{p,c_1} \circ \varphi^{-1}) = Q_{p, \varphi(c_1)}.
	\end{align*}
	Hence, we have that $\varphi^{-1}\left(\mT(\im{r},\im{s},\im{q})\right)=\mT(\im{m},\im{k},\im{l})$. In addition, let $c_2\in\mT(\im{r},\im{q},\im{s})$. Using the equality above with $c_1 = \varphi^{-1}(c_2)$, we find that
	\begin{align*}
	\left(\varphi^{-1} \circ Q_{p,c_2} \circ \varphi\right) &= \left(\varphi^{-1} \circ (\varphi \circ Q_{p,\varphi^{-1}(c_2)} \circ \varphi^{-1} ) \circ \varphi\right) = Q_{p,\varphi^{-1}(c_2)}.
	\end{align*}
	Therefore, we found a suitable bijective linear application $\varphi^{-1}$ to conclude that $\cT^p_2\sim\cT^p_1$.
	
	\item Let $\cT^p_1(\im{m},\im{k},\im{l})$, $\cT^p_2(\im{r},\im{q},\im{s})$ and $\cT^p_3(\im{t},\im{u},\im{v})$ be three principal 3D slices such that $\cT^p_1\sim\cT_2$ and $\cT^p_2\sim\cT^p_3$. In addition, consider the sets $M_1~=~\iter^p(\im{m},\im{k},\im{l})$, $M_2 = \iter^p(\im{r},\im{q},\im{s})$ and $M_3 = \iter^p(\im{t},\im{u},\im{v})$. We know that there exists bijective linear applications $\varphi_1:M_1\rightarrow M_2$ and $\varphi_2:M_2\rightarrow M_3$ which are conform to the hypotheses in Definition \ref{defEquivRel}. Let $\Phi:M_1\rightarrow M_3$ such that $\Phi = \varphi_2\circ\varphi_1$. We can see that
	\begin{align*}	
	\Phi\left(\mT(\im{m},\im{k},\im{l})\right) &= \varphi_2\left(\varphi_1\left(\mT(\im{m},\im{k},\im{l})\right)\right),\\
	&= \varphi_2\left(\mT(\im{r},\im{q},\im{s})\right),\\
	&= \mT(\im{t},\im{u},\im{v}).
	\end{align*}
	Moreover, we have that, for all $c\in\mT(\im{m},\im{k},\im{l})$,
	\begin{align*}
	(\Phi\circ Q_{p,c}\circ\Phi^{-1}) &= (\varphi_2\circ\varphi_1\circ Q_{p,c}\circ\varphi_1^{-1}\circ\varphi_2^{-1}),\\
	&= (\varphi_2\circ Q_{p,\varphi_1(c)}\circ\varphi_2^{-1}),\\
	&= Q_{p,\varphi_2(\varphi_1(c))} = Q_{p,\Phi(c)}.
	\end{align*}
	Thus, we conclude that $\cT^p_1\sim\cT^p_3$.\qedhere
\end{itemize}
\end{proof}

\section{Characterization of the 3D Slices}

It is possible to determine the nature of $\iter^p(\im{m},\im{k},\im{l})$. By doing this, it will become easier to find similarities between the associated principal 3D slices.

\begin{lemma}\label{lemCharacterization}
Let $\im{m},\im{k},\im{l}\in\mI(n)$. We define the following vector subspaces of $\mM(n)$:
\begin{align*}
\mM(\im{m},\im{k},\im{l}) &= \vspan_{\mR}\left\lbrace \im{m},\, \im{k},\, \im{l},\, \im{m}\im{k}\im{l} \right\rbrace\!;\\
\mS(\im{m},\im{k},\im{l}) &= \vspan_{\mR}\{1,\, \im{m},\, \im{k},\,\im{l},\, \im{m}\im{k},\, \im{m}\im{l},\, \im{k}\im{l},\, \im{m}\im{k}\im{l}\}.
\end{align*}
The subspaces $\mM(\im{m},\im{k},\im{l})$ and $\mS(\im{m},\im{k},\im{l})$ are closed under the addition. Moreover, we have that
\begin{enumerate}
\item the subspace $\mM(1,\im{k},\im{l})$ is closed under the multiplication;
\item the subspace $\mS(\im{m},\im{k},\im{l})$ is also closed under the multiplication;
\item when $p$ is odd, if $\eta\in\mM(\im{m},\im{k},\im{l})$, then $\eta^p\in\mM(\im{m},\im{k},\im{l})$.
\end{enumerate}
\end{lemma}

\begin{proof}
As vector spaces, the closure of $\mM(\im{m},\im{k},\im{l})$ and $\mS(\im{m},\im{k},\im{l})$ under the addition is obvious. The other three statements may be verified algebraically.

\begin{enumerate}
	\item Consider $\eta,\zeta\in\mM(1,\im{k},\im{l})$ such that
	\begin{align*}
	\eta = x_1 + x_2\im{k} + x_3\im{l} + x_4\im{k}\im{l}\quad\text{and}\quad \zeta = y_1 + y_2\im{k} + y_3\im{l} + y_4\im{k}\im{l}.
	\end{align*}
	We find that $\eta\cdot\zeta\in\mM(1,\im{k},\im{l})$ since
	\begin{align*}
	\eta\cdot\zeta &= (x_1 + x_2\im{k} + x_3\im{l} + x_4\im{k}\im{l})(y_1 + y_2\im{k} + y_3\im{l} + y_4\im{k}\im{l}),\\
	&= x_1y_1 + x_2y_2\im{k}^2 + x_3y_3\im{l}^2 + x_4y_4\im{k}^2\im{l}^2\\
	&\quad + (x_1y_2 + x_2y_1 + x_3y_4\im{l}^2 + x_4y_3\im{l}^2)\im{k}\\
	&\quad + (x_1y_3 + x_3y_1 + x_2y_4\im{k}^2 + x_4y_2\im{k}^2)\im{l}\\
	&\quad + (x_2y_3 + x_3y_2 + x_1y_4 + x_4y_1)\im{k}\im{l}.
	\end{align*}

	\item For all $\eta,\zeta\in \mS(\im{m},\im{k},\im{l})$, it can be shown that $\eta\cdot\zeta\in \mS(\im{m},\im{k},\im{l})$ similarly to the previous case.

	\item Now, let $\eta\in\mM(\im{m},\im{k},\im{l})$ such that
	\[\eta = x_1\im{m} + x_2\im{k} + x_3\im{l} + x_4\im{m}\im{k}\im{l}\]
	and consider $\mu\in\mM(\im{m},\im{k},\im{l})$ and $\nu\in\mM(\im{m}\im{k}, \im{m}\im{l}, \im{k}\im{l})$ such that
	\[\mu = y_1\im{m} + y_2\im{k} + y_3\im{l} + y_4\im{m}\im{k}\im{l}\quad\text{and}\quad \nu = w_1 + w_2\im{m}\im{k} + w_3\im{m}\im{l} + w_4\im{k}\im{l}.\]

	Notice that
	\begin{align*}
	\eta\cdot\mu &= x_1y_1\im{m}^2 + x_2y_2\im{k}^2 + x_3y_3\im{l}^2 + x_4y_4\im{m}^2\im{k}^2\im{l}^2\\
	&\quad + (x_1y_2 + x_2y_1 + x_3y_4\im{l}^2 + x_4y_3\im{l}^2)\im{m}\im{k}\\
	&\quad + (x_1y_3 + x_3y_1 + x_2y_4\im{k}^2 + x_4y_2\im{k}^2)\im{m}\im{l}\\
	&\quad + (x_2y_3 + x_3y_2 + x_1y_4\im{m}^2 + x_4y_1\im{m}^2)\im{k}\im{l}
	\end{align*}
	and
	\begin{align*}
	\eta\cdot\nu &= (x_1w_1 + x_2w_2\im{k}^2 + x_3w_3\im{l}^2 + x_4w_4\im{k}^2\im{l}^2)\im{m}\\
	&\quad + (x_2w_1 + x_1w_2\im{m}^2 + x_3w_4\im{l}^2 + x_4w_3\im{m}^2\im{l}^2)\im{k}\\
	&\quad + (x_3w_1 + x_1w_3\im{m}^2 + x_2w_4\im{k}^2 + x_4w_2\im{m}^2\im{k}^2)\im{l}\\
	&\quad + (x_4w_1 + x_1w_4 + x_2w_3 + x_3w_2)\im{m}\im{k}\im{l}.
	\end{align*}
	meaning that $\eta\cdot\mu\in\mM(\im{m}\im{k},\im{m}\im{l},\im{k}\im{l})$ and $\eta\cdot\nu\in\mM(\im{m},\im{k},\im{l})$. Using these two arguments, we see that
	\[\eta^2 = \eta\cdot\eta\in\mM(\im{m}\im{k},\im{m}\im{l},\im{k}\im{l})\quad\text{and}\quad \eta^3=\eta\cdot\eta^2\in\mM(\im{m},\im{k},\im{l}).\]
	Then, it can be verified that $\eta^p\in\mM(\im{m},\im{k},\im{l})$ when $p$ is odd using the induction principle. \qedhere
\end{enumerate}
\end{proof}

\begin{lemma}\label{lemDimensions}
Let $\im{m},\im{k},\im{l}\in\mI(n)$ and consider $M = \iter^p(\im{m},\im{k},\im{l})$. Then,
\begin{enumerate}
	\item if $p$ is even and $\im{m} = 1$ or $\im{k}\im{l} = \pm \im{m}$, then $\dim(M) \geq 4$;
	\item if $p$ is even but $\im{m} \neq 1$ and $\im{k}\im{l} \neq \pm\im{m}$, then $\dim(M) \geq 8$;
	\item if $p$ is odd, then $\dim(M)\geq 4$.
\end{enumerate}
\end{lemma}

\begin{remark}
When one of the units $\im{m}$, $\im{k}$ or $\im{l}$ is 1, we suppose without loss of generality that $\im{m}=1$. Thus, $\im{m}\neq 1$ implies that none of the three units is 1. Therefore, the three cases above cover all possible subspaces $\iter^p(\im{m},\im{k},\im{l})$.
\end{remark}

\begin{proof}
Consider the vector space $M = \iter^p(\im{m},\im{k},\im{l})$. Basically, we have to find four or eight linearly independent vectors, depending on the case considered. In general, we have the three vectors $Q_{p,\im{m}}(0) = \im{m}$, $Q_{p,\im{k}}(0) = \im{k}$ and $Q_{p,\im{l}}(0) = \im{l}$. Then, we have to consider each case separately.
\begin{enumerate}
	\item When $p$ is even and $\im{m} = 1$ or $\im{k}\im{l} = \pm \im{m}$, consider $c_1 = a_1\im{k}+\im{l}$ with $a_1\in\mR^*$. It is possible to find some $a_1\in\mR^*$ such that
	\[(a_1\im{k}+\im{l})^p = x_{11} + x_{12}\im{k}\im{l}\]
	with non-zero values $x_{11},x_{12}\in\mR^*$. Indeed, using the binomial Theorem, we can calculate that
	\begin{align*}
	(a_1\im{k}+\im{l})^p &= \sum_{j=0}^p \binom{p}{j}a_1^j\im{k}^j\im{l}^{p-j},\\
	&= \sum_{j=0}^{\frac{p}{2}}\binom{p}{2j}a_1^{2j}(\im{k}^2)^j(\im{l}^2)^{\frac{p}{2}-j}\\
	&\quad + \im{k}\im{l}\sum_{j=0}^{\frac{p}{2}-1}\binom{p}{2j+1}a_1^{2j+1}(\im{k}^2)^j(\im{l}^2)^{\frac{p}{2}-j-1}.
	\end{align*}
	One may notice that these last two sums are real polynomials in $a_1$ which have degrees of $p$ and $p-1$ respectively. Hence, there are at most $2p-1$ values of $a_1$ such that one of the two sums is zero. This implies that there exists an infinite number of real values $a_1$ such that both sums are non-zero.
	
	Thus, consider $a_1\in\mR$ such that $c_1^p = (a_1\im{k}+\im{l})^p = x_{11} + x_{12}\im{k}\im{l}$ where $x_{11},x_{12}\in\mR^*$ are non-zero. We have that
	\[Q^2_{p,c_1}(0) = c_1^p + c_1 = x_{11} + a_1\im{k} + \im{l} + x_{12}\im{k}\im{l}.\]
	Because the values $Q_{p,\im{m}}(0)$, $Q_{p,\im{k}}(0)$, $Q_{p,\im{l}}(0)$ and $Q^2_{p,c_1}(0)$ are four linearly independent vectors of $M$, it follows that $\dim(M)\geq 4$.
	
	\item When $p$ is even but $\im{m} \neq 1$ and $\im{k}\im{l} \neq \pm\im{m}$, consider the constant $c_1$ above as well as $c_2 = a_2\im{m}+\im{k}$ and $c_3 = a_3\im{m}+\im{l}$ where $a_2,a_3\in\mR^*$ are such that $c_2^p = x_{21} + x_{22}\im{m}\im{k}$ and $c_3^p = x_{31} + x_{32}\im{m}\im{l}$ with non-zero values $x_{21},x_{22},x_{31},x_{32}\in\mR^*$. We can make sure that such constants exist by using a reasoning similar to the one used previously for $c_1$. Then, we have that
	\begin{align*}
	Q^2_{p,\im{m}}(0) &= \pm 1 + \im{m},\\
	Q^2_{p,c_1}(0) &= x_{11} + a_1\im{k} + \im{l} + x_{12}\im{k}\im{l},\\
	Q^2_{p,c_2}(0) &= x_{21} + a_2\im{m} + \im{k} + x_{22}\im{m}\im{k},\\
	Q^2_{p,c_3}(0) &= x_{31} + a_3\im{m} + \im{l} + x_{32}\im{m}\im{l}.
	\end{align*}
	Let $c_0 = \im{m} + \im{k} + \im{l}$. Then, we see that
	\[Q^3_{p,c_0}(0) = y_1 + y_2\im{m} + y_3\im{k} + y_4\im{l} + y_5\im{m}\im{k} + y_6\im{m}\im{l} + y_7\im{k}\im{l} + y_8\im{m}\im{k}\im{l}\]
	for values $y_j\in\mR$ with $y_8\neq 0$. Therefore, considering these last five iterates in addition to the three iterates $Q_{p,\im{m}}(0)$, $Q_{p,\im{k}}(0)$ and $Q_{p,\im{l}}(0)$, we see that there exists at least eight linearly independent vectors in $M$, meaning that $\dim(M)\geq 8$.
	
	\item When $p$ is odd, consider $c_0 = \im{m} + \im{k} + \im{l}$. Using Lemma \ref{lemCharacterization}, It can be computed that
	\[Q^2_{p,c_0} = u_1\im{m} + u_2\im{k} + u_3\im{l} + u_4\im{m}\im{k}\im{l}\]
	with $u_j\in\mR$ and $u_4\neq 0$. Therefore, we have found four linearly independent vectors in $M$, hence $\dim(M)\geq 4$.\qedhere
\end{enumerate}
\end{proof}

Using these lemmas, the subspace $\iter^p(\im{m},\im{k},\im{l})$ can now be characterized.

\begin{theorem}\label{theoCharacterization}
Let $\im{m},\im{k},\im{l}\in\mI(n)$ and $M=\iter^p(\im{m},\im{k},\im{l})$. We have that
\begin{enumerate}
	\item if $p$ is even and $\im{m} = 1$ or $\im{k} \im{l} = \pm \im{m}$, then $M = \mM (1,\im{k},\im{l})$;
	\item if $p$ is even but $\im{m} \neq 1$ and $\im{k}\im{l} \neq \pm\im{m}$, then $M = \mS(\im{m},\im{k},\im{l})$;
	\item if $p$ is odd, then $M = \mM (\im{m},\im{k},\im{l})$.
\end{enumerate}
\end{theorem}
	
\begin{proof}
Let $M = \iter^p(\im{m},\im{k},\im{l})$.
\begin{enumerate}
	\item Suppose that $p$ is even. First, notice that $Q_{p,c}(0) = c \in\mM(\im{m},\im{k},\im{l})$ for all $c\in\mT(\im{m},\im{k},\im{l})$.	Moreover, if $\im{m} = 1$ or $\im{k} \im{l} = \pm \im{m}$, we have that $\mM(\im{m},\im{k},\im{l})=\mM (1 , \im{k} , \im{l} )$, which is closed under the addition and multiplication operations according to Lemma \ref{lemCharacterization}. Therefore, we can show that $Q_{p,c}^m (0) \in \mM (1, \im{k} , \im{l} )$ $\forall m \in \mN$, meaning that $M\subseteq \mM(1,\im{k},\im{l})$.
	
	Next, to prove that $M=\mM(1,\im{k},\im{l})$, we use some linear algebra concepts. From Lemma \ref{lemDimensions}, we know that $\dim(M)\geq 4$. Since we have that $M~\subseteq~\mM(1,\im{k},\im{l})$, we see that $\dim(M)\leq \dim\left(\mM(1,\im{k},\im{l})\right)=4$. Thus, it follows that $\dim(M) = 4 = \dim\left(\mM(1,\im{k},\im{l})\right)$. Since $M$ is a subspace of the vector space $\mM(1,\im{k},\im{l})$ and both spaces have the same finite dimension, we conclude that $M = \mM(1,\im{k},\im{l})$.
	
	\item Suppose that $p$ is even and $\im{m} \neq 1$, $\im{k} \im{l} \neq \pm \im{m}$. First, we see that $Q_{p,c}(0)\in \mS(\im{m},\im{k},\im{l})$. Then, we find that $Q^m_{p,c}(0)\in \mS(\im{m},\im{k},\im{l})\ \forall m\in\mN$ since, according to Lemma \ref{lemCharacterization}, $\mS(\im{m},\im{k},\im{l})$ is closed under the addition and the multiplication, meaning that $M\subseteq \mS(\im{m},\im{k},\im{l})$.
	
	It follows that $\dim(M) \leq \dim(\mS(\im{m},\im{k},\im{l}))$. From Lemma \ref{lemDimensions}, we know that $\dim(M) \geq 8 = \dim(\mS(\im{m},\im{k},\im{l}))$. Therefore, we find that $\dim(M) = \dim(\mS(\im{m},\im{k},\im{l}))$. Since $M\subseteq \mS(\im{m},\im{k},\im{l})$ and both spaces have the same finite dimension, we have that $M = \mS(\im{m},\im{k},\im{l})$.
	
	\item Suppose that $p$ is odd. Again, we have that $Q_{p,c}(0)\in\mM(\im{m},\im{k},\im{l})$. Moreover, we know from Lemma \ref{lemCharacterization} that $\mM(\im{m},\im{k},\im{l})$ is closed under the addition and $\eta^p\in\mM(\im{m},\im{k},\im{l})$ for all $\eta\in\mM(\im{m},\im{k},\im{l})$. Thus, we can show that $Q_{p,c}^m (0)\in\mM(\im{m},\im{k},\im{l})\ \forall m\in\mN$, hence $M\subseteq \mM(\im{m},\im{k},\im{l})$.
		
	This implies that $\dim(M) \leq \dim(\mM(\im{m},\im{k},\im{l})$. Again, we know from Lemma \ref{lemDimensions} that $\dim(M) \geq 4 = \dim(\mM(\im{m},\im{k},\im{l}))$. Consequently, we have that $\dim(M) = \dim(\mM(\im{m},\im{k},\im{l}))$. Since $M\subseteq \mM(\im{m},\im{k},\im{l})$, we conclude that $M = \mM(\im{m},\im{k},\im{l})$.\qedhere
\end{enumerate}
\end{proof}

\section{An Optimal Level of Generalization}

Even though there are many ways to choose $\im{m}$, $\im{k}$ and $\im{l}$, Theorem \ref{theoCharacterization} indicates that $\iter^p(\im{m},\im{k},\im{l})$, and consequently $\cT^p(\im{m},\im{k},\im{l})$, may only behave in a limited number of ways.

This is, basically, the reason why we are able to prove that any multicomplex principal 3D slice is equivalent to a tricomplex slice up to an affine transformation. Before presenting the complete proof, we explain here the summarized idea.

In short, we need to find units $\im{r},\im{q},\im{s}\in\mI(3)$ and a bijective linear application $\varphi:M_1\rightarrow M_2$ which is conform to the hypotheses of Definition \ref{defEquivRel}. Using Theorem \ref{theoCharacterization}, it is fairly easy to define an appropriate application $\varphi$ which depends on the hypotheses regarding $p$ and the units $\im{m}$, $\im{k}$ and $\im{l}$. However, the more arduous part is then to make sure that, in all three cases presented in Theorem \ref{theoCharacterization}, the equality $\varphi(\eta^p)=\varphi(\eta)^p$ holds $\forall\eta\in M_1$. Afterwards, it follows directly that
\begin{align*}
(\varphi \circ Q_{p,\varphi^{-1}(c)} \circ \varphi^{-1}) (\eta) = Q_{p, c} (\eta )
\end{align*}
for all $c\in\mT(\im{r},\im{q},\im{s})$ and for all $\eta \in M_2$, hence the result. We will see that there is one specific case where units $\im{r}$, $\im{q}$ and $\im{s}$ have to be quadricomplex. Still, in that case, the quadricomplex principal 3D slice can be obtained by applying an affine transformation on a tricomplex slice.

\begin{lemma}\label{lemEtap}
Let $\im{m},\im{k},\im{l}\in\mI(n)$ and $\im{r},\im{q},\im{s}\in\mI(n)$ such that $\im{m}^2=\im{r}^2$, $\im{k}^2=\im{q}^2$ and $\im{l}^2=\im{s}^2$. Moreover, consider $M_1 = \iter^p(\im{m},\im{k},\im{l})$ and $M_2 = \iter^p(\im{r},\im{q},\im{s})$.
\begin{enumerate}
	\item If $p$ is even and $\im{m} = 1$ or $\im{k}\im{l} = \pm \im{m}$, assume $\im{r} = 1$ or $\im{q}\im{s} = \pm\im{r}$ respectively. Then, since $M_1=\mM(1,\im{k},\im{l})$ and $M_2=\mM(1,\im{q},\im{s})$, we define
	\[\varphi(x_1 + x_2\im{k} + x_3\im{l} + x_4\im{k}\im{l}) = x_1 + x_2\im{q} + x_3\im{s} + x_4\im{q}\im{s}.\]

	\item If $p$ is even but $\im{m} \neq 1$ and $\im{k} \im{l} \neq \pm \im{m}$, assume $\im{r} \neq 1$ and $\im{q}\im{s} \neq \pm\im{r}$. Since $M_1 = \mS(\im{m},\im{k},\im{l})$ and $M_2 = \mS(\im{r},\im{q},\im{s})$, we can define
	\begin{align*}
	\varphi(x_1 + x_2\im{m} + x_3\im{k} + x_4\im{l} + x_5\im{m}\im{k} + x_6\im{m}\im{l} + x_7\im{k}\im{l} + x_8\im{m}\im{k}\im{l})&\\
	= x_1 + x_2\im{r} + x_3\im{q} + x_4\im{s} + x_5\im{r}\im{q} + x_6\im{r}\im{s} + x_7\im{q}\im{s} + x_8\im{r}\im{q}\im{s}.&
	\end{align*}

	\item If $p$ is odd, then $M_1 = \mM(\im{m},\im{k},\im{l})$ and $M_2 = \mM(\im{r},\im{q},\im{s})$. Thus, consider
	\[\varphi(x_1\im{m} + x_2\im{k} + x_3\im{l} + x_4\im{m}\im{k}\im{l}) = x_1\im{r} + x_2\im{q} + x_3\im{s} + x_4\im{r}\im{q}\im{s}.\]
\end{enumerate}
In each of these cases, $\varphi$ is a linear bijective application such that, for all $\eta\in M_1$, $\varphi(\eta^p) = \varphi(\eta)^p$.
\end{lemma}

\begin{proof}
It can easily be seen that $\varphi$ is a linear bijective application in each case. Moreover, we know from Lemma \ref{lemCharacterization} that $\eta^p\in M_1$ for all $\eta\in M_1$, meaning that $\varphi(\eta^p)$ is always defined. We must now prove the equation $\varphi(\eta^p) = \varphi(\eta)^p$ holds $\forall \eta\in M_1$.
\begin{enumerate}
	\item In the first case, we can verify that $\varphi(\eta\cdot\zeta) = \varphi(\eta)\varphi(\zeta)\ \forall \eta,\zeta\in M_1$. Indeed, if
	\[\eta = x_1 + x_2\im{k} + x_3\im{l} + x_4\im{k}\im{l}\quad\text{and}\quad \zeta = y_1 + y_2\im{k} + y_3\im{l} + y_4\im{k}\im{l},\]
	then
	\[\varphi(\eta) = x_1 + x_2\im{q} + x_3\im{s} + x_4\im{q}\im{s}\quad\text{and}\quad \varphi(\zeta) = y_1 + y_2\im{q} + y_3\im{s} + y_4\im{q}\im{s}.\]
	Since $\im{k}^2=\im{q}^2$ and $\im{l}^2=\im{s}^2$,
	\begin{align*}
	\varphi(\eta\cdot\zeta) &= \varphi\Big(x_1y_1 + x_2y_2\im{k}^2 + x_3y_3\im{l}^2 + x_4y_4\im{k}^2\im{l}^2\\
	&\quad + (x_1y_2 + x_2y_1 + x_3y_4\im{l}^2 + x_4y_3\im{l}^2)\im{k}\\
	&\quad + (x_1y_3 + x_3y_1 + x_2y_4\im{k}^2 + x_4y_2\im{k}^2)\im{l}\\
	&\quad + (x_2y_3 + x_3y_2 + x_1y_4 + x_4y_1)\im{k}\im{l}\Big),\\
	&= x_1y_1 + x_2y_2\im{q}^2 + x_3y_3\im{s}^2 + x_4y_4\im{q}^2\im{s}^2\\
	&\quad + (x_1y_2 + x_2y_1 + x_3y_4\im{s}^2 + x_4y_3\im{s}^2)\im{q}\\
	&\quad + (x_1y_3 + x_3y_1 + x_2y_4\im{q}^2 + x_4y_2\im{q}^2)\im{s}\\
	&\quad + (x_2y_3 + x_3y_2 + x_1y_4 + x_4y_1)\im{q}\im{s},\\
	&= \varphi(\eta)\varphi(\zeta).
	\end{align*}
	Consequently, we find that $\varphi(\eta^p)=\varphi(\eta)^p$.
	
	\item The second case is analogous to the first. Again, we verify, with more laborious but very similar calculations, that $\varphi(\eta\cdot\zeta) = \varphi(\eta)\varphi(\zeta)$ for all $\eta,\zeta\in M_1$, hence $\varphi(\eta^p)=\varphi(\eta)^p$.
	
	\item The last case cannot be treated like the two first ones since the product of two numbers $\eta,\zeta\in M_1$ is not necessarily in $M_1$, meaning that $\varphi(\eta\cdot\zeta)$ is not always defined. Nevertheless, using the induction principle on odd integers $p\geq 3$, the same conclusion can be found.
	
	Let $\eta\in M_1$ such that $\eta = x_1\im{m} + x_2\im{k} + x_3\im{l} + x_4\im{m}\im{k}\im{l}$. Then, $\varphi(\eta) = x_1\im{r} + x_2\im{q} + x_3\im{s} + x_4\im{r}\im{q}\im{s}$. When $p = 3$, since $\im{m}^2=\im{r}^2$, $\im{k}^2=\im{q}^2$ and $\im{l}^2=\im{s}^2$, we find that
	\begin{align*}
	\varphi(\eta^3) &= \varphi\Big((x_1\im{m} + x_2\im{k} + x_3\im{l} + x_4\im{m}\im{k}\im{l})^3\Big),\\
	&= \varphi\Big((x_1^3\im{m}^2 + 3x_1x_2^2\im{k}^2 + 3x_1x_3^2\im{l}^2 + 3x_1x_4^2\im{m}^2\im{k}^2\im{l}^2 + 6x_2x_3x_4\im{k}^2\im{l}^2)\im{m}\\
	&\quad + (x_2^3\im{k}^2 + 3x_1^2x_2\im{m}^2 + 3x_2x_3^2\im{l}^2 + 3x_2x_4^2\im{m}^2\im{k}^2\im{l}^2 + 6x_1x_3x_4\im{m}^2\im{l}^2)\im{k}\\
	&\quad + (x_3^3\im{l}^2 + 3x_1^2x_3\im{m}^2 + 3x_2^2x_3\im{k}^2 + 3x_3x_4^2\im{m}^2\im{k}^2\im{l}^2 + 6x_1x_2x_4\im{m}^2\im{k}^2)\im{l}\\
	&\quad + (x_4^3\im{m}^2\im{k}^2\im{l}^2 + 3x_1^2x_4\im{m}^2 + 3x_2^2x_4\im{k}^2 + 3x_3^2x_4\im{l}^2 + 6x_1x_2x_3)\im{m}\im{k}\im{l}\Big), \displaybreak[0]\\
	&= (x_1^3\im{r}^2 + 3x_1x_2^2\im{q}^2 + 3x_1x_3^2\im{s}^2 + 3x_1x_4^2\im{r}^2\im{q}^2\im{s}^2 + 6x_2x_3x_4\im{q}^2\im{s}^2)\im{r}\\
	&\quad + (x_2^3\im{q}^2 + 3x_1^2x_2\im{r}^2 + 3x_2x_3^2\im{s}^2 + 3x_2x_4^2\im{r}^2\im{q}^2\im{s}^2 + 6x_1x_3x_4\im{r}^2\im{s}^2)\im{q}\\
	&\quad + (x_3^3\im{s}^2 + 3x_1^2x_3\im{r}^2 + 3x_2^2x_3\im{q}^2 + 3x_3x_4^2\im{r}^2\im{q}^2\im{s}^2 + 6x_1x_2x_4\im{r}^2\im{q}^2)\im{s}\\
	&\quad + (x_4^3\im{r}^2\im{q}^2\im{s}^2 + 3x_1^2x_4\im{r}^2 + 3x_2^2x_4\im{q}^2 + 3x_3^2x_4\im{s}^2 + 6x_1x_2x_3)\im{r}\im{q}\im{s},\\
	&= \varphi(\eta)^3.
	\end{align*}
	Now, suppose that $\varphi(\eta^p) = \varphi(\eta)^p$ for some odd integer $p\geq 3$. Consider that
	\[\eta^p = y_1\im{m} + y_2\im{k} + y_3\im{l} + y_4\im{m}\im{k}\im{l}.\]
	Then,
	\[\varphi(\eta^p) = y_1\im{r} + y_2\im{q} + y_3\im{s} + y_4\im{r}\im{q}\im{s}.\]
	We can calculate that
	\begin{align*}
	\eta^{p+2} &= \eta^p\cdot\Big((x_1^2\im{m}^2 + x_2^2\im{k}^2 + x_3^2\im{l}^2 + x_4^2\im{m}^2\im{k}^2\im{l}^2) + (2x_1x_2 + 2x_3x_4\im{l}^2)\im{m}\im{k}\\
	&\quad + (2x_1x_3 + 2x_2x_4\im{k}^2)\im{m}\im{l} + (2x_2x_3 + 2x_1x_4\im{m}^2)\im{k}\im{l}\Big), \displaybreak[0]\\
	&= \Big(y_1(x_1^2\im{m}^2 + x_2^2\im{k}^2 + x_3^2\im{l}^2 + x_4^2\im{m}^2\im{k}^2\im{l}^2) + y_2(2x_1x_2 + 2x_3x_4\im{l}^2)\im{k}^2\\
	&\quad + y_3(2x_1x_3 + 2x_2x_4\im{k}^2)\im{l}^2 + y_4(2x_2x_3 + 2x_1x_4\im{m}^2)\im{k}^2\im{l}^2\Big)\im{m}\\
	&\quad + \Big(y_2(x_1^2\im{m}^2 + x_2^2\im{k}^2 + x_3^2\im{l}^2 + x_4^2\im{m}^2\im{k}^2\im{l}^2) + y_1(2x_1x_2 + 2x_3x_4\im{l}^2)\im{m}^2\\
	&\quad + y_4(2x_1x_3 + 2x_2x_4\im{k}^2)\im{m}^2\im{l}^2 + y_3(2x_2x_3 + 2x_1x_4\im{m}^2)\im{l}^2\Big)\im{k}\\
	&\quad + \Big(y_3(x_1^2\im{m}^2 + x_2^2\im{k}^2 + x_3^2\im{l}^2 + x_4^2\im{m}^2\im{k}^2\im{l}^2) + y_4(2x_1x_2 + 2x_3x_4\im{l}^2)\im{m}^2\im{k}^2\\
	&\quad + y_1(2x_1x_3 + 2x_2x_4\im{k}^2)\im{m}^2 + y_2(2x_2x_3 + 2x_1x_4\im{m}^2)\im{k}^2\Big)\im{l}\\
	&\quad + \Big(y_4(x_1^2\im{m}^2 + x_2^2\im{k}^2 + x_3^2\im{l}^2 + x_4^2\im{m}^2\im{k}^2\im{l}^2) + y_3(2x_1x_2 + 2x_3x_4\im{l}^2)\\
	&\quad + y_2(2x_1x_3 + 2x_2x_4\im{k}^2) + y_1(2x_2x_3 + 2x_1x_4\im{m}^2)\Big)\im{m}\im{k}\im{l}.
	\end{align*}
	Therefore, since $\im{m}^2=\im{r}^2$, $\im{k}^2=\im{q}^2$ and $\im{l}^2=\im{s}^2$,
	\begin{align*}
	\varphi(\eta^{p+2}) &= \Big(y_1(x_1^2\im{r}^2 + x_2^2\im{q}^2 + x_3^2\im{s}^2 + x_4^2\im{r}^2\im{q}^2\im{s}^2) + y_2(2x_1x_2 + 2x_3x_4\im{s}^2)\im{q}^2\\
	&\quad + y_3(2x_1x_3 + 2x_2x_4\im{q}^2)\im{s}^2 + y_4(2x_2x_3 + 2x_1x_4\im{r}^2)\im{q}^2\im{s}^2\Big)\im{r}\\
	&\quad + \Big(y_2(x_1^2\im{r}^2 + x_2^2\im{q}^2 + x_3^2\im{s}^2 + x_4^2\im{r}^2\im{q}^2\im{s}^2) + y_1(2x_1x_2 + 2x_3x_4\im{s}^2)\im{r}^2\\
	&\quad + y_4(2x_1x_3 + 2x_2x_4\im{q}^2)\im{r}^2\im{s}^2 + y_3(2x_2x_3 + 2x_1x_4\im{r}^2)\im{s}^2\Big)\im{q}\\
	&\quad + \Big(y_3(x_1^2\im{r}^2 + x_2^2\im{q}^2 + x_3^2\im{s}^2 + x_4^2\im{r}^2\im{q}^2\im{s}^2) + y_4(2x_1x_2 + 2x_3x_4\im{s}^2)\im{r}^2\im{q}^2\\
	&\quad + y_1(2x_1x_3 + 2x_2x_4\im{q}^2)\im{r}^2 + y_2(2x_2x_3 + 2x_1x_4\im{r}^2)\im{q}^2\Big)\im{s}\\
	&\quad + \Big(y_4(x_1^2\im{r}^2 + x_2^2\im{q}^2 + x_3^2\im{s}^2 + x_4^2\im{r}^2\im{q}^2\im{s}^2) + y_3(2x_1x_2 + 2x_3x_4\im{s}^2)\\
	&\quad + y_2(2x_1x_3 + 2x_2x_4\im{q}^2) + y_1(2x_2x_3 + 2x_1x_4\im{r}^2)\Big)\im{r}\im{q}\im{s}, \displaybreak[0] \\
	&= \varphi(\eta^p)\cdot\Big((x_1^2\im{r}^2 + x_2^2\im{q}^2 + x_3^2\im{s}^2 + x_4^2\im{r}^2\im{q}^2\im{s}^2) + (2x_1x_2 + 2x_3x_4\im{s}^2)\im{r}\im{q}\\
	&\quad + (2x_1x_3 + 2x_2x_4\im{q}^2)\im{r}\im{s} + (2x_2x_3 + 2x_1x_4\im{r}^2)\im{q}\im{s}\Big),\\
	&= \varphi(\eta)^p\cdot\varphi(\eta)^2 = \varphi(\eta)^{p+2}.
	\end{align*}
	Hence, when $p\geq 3$ is odd, the equation $\varphi(\eta^p)=\varphi(\eta)^p$ holds $\forall\eta\in M_1$.
\end{enumerate}
Therefore, we conclude that the equality $\varphi(\eta^p)=\varphi(\eta)^p$ is valid in the three cases presented.
\end{proof}

Using this lemma, we would like to show that, whichever multicomplex slice $\cT^p_1(\im{m},\im{k},\im{l})$ we consider, there exists a tricomplex slice $\cT^p_2(\im{r},\im{q},\im{s})$ with the same dynamics. However, there is one marginal case. Suppose $p$ is even and consider $\im{m} = \ibasic{1}\ibasic{2}$, $\im{k} = \ibasic{1}\ibasic{3}$ and $\im{l} = \ibasic{1}\ibasic{4}$ where $\ibasic{1}$, $\ibasic{2}$, $\ibasic{3}$ and $\ibasic{4}$ are four basic complex imaginary units. We see that all three units are hyperbolic but $\im{k}\im{l}\neq\pm\im{m}$. In $\mI(3)$, there are only three hyperbolic units and they are such that $\im{k}\im{l} = \pm\im{m}$, meaning that the dynamics of $\cT^p(\ibasic{1}\ibasic{2},\ibasic{1}\ibasic{3},\ibasic{1}\ibasic{4})$ cannot be replicated in $\mM(3)$. Notice that this is only a problem when $p$ is even since, as seen in the previous results, the relation between $\im{k}\im{l}$ and $\im{m}$ does not have an impact the dynamics of a slice when $p$ is odd.

Nonetheless, we can show that this particular slice, when represented in the tridimensional space, is an octahedron just like the Perplexbrot (also called the Airbrot) $\cT^p(1,\ibasic{1}\ibasic{2},\ibasic{1}\ibasic{3})$ found in the tricomplex space \cite{GarantPelletier, GarantRochon, RochonRansfordParise}.

\begin{lemma}\label{lemOcta}
Let $\im{m} = \ibasic{1}\ibasic{2}$, $\im{k} = \ibasic{1}\ibasic{3}$ and $\im{l} = \ibasic{1}\ibasic{4}$ where $\ibasic{1}$, $\ibasic{2}$, $\ibasic{3}$ and $\ibasic{4}$ are four basic complex imaginary units. The principal 3D slice $\cT^p(\im{m},\im{k},\im{l})$ is the following regular octahedron:
\[\cT^p(\im{m},\im{k},\im{l}) = \left\lbrace c_1\ibasic{1}\ibasic{2} + c_2\ibasic{1}\ibasic{3} + c_3\ibasic{1}\ibasic{4}\ :\ |c_1| + |c_2| + |c_3| \leq \frac{p-1}{p^{p/(p-1)}} \right\rbrace.\]
\end{lemma}

\begin{proof}
Consider $c = c_1\ibasic{1}\ibasic{2} + c_2\ibasic{1}\ibasic{3} + c_3\ibasic{1}\ibasic{4}$. Using the idempotent representation (see Subsection \ref{secIdempotent}), we can express $c$ as
\begin{align*}
c &= (c_1 - c_2 + c_3)\gm{1}\gm{2}\gm{3} + (- c_1 + c_2 - c_3)\gmc{1}\gm{2}\gm{3},\\
&\quad + (c_1 + c_2 - c_3)\gm{1}\gmc{2}\gm{3} + (- c_1 - c_2 + c_3)\gmc{1}\gmc{2}\gm{3},\\
&\quad + (c_1 - c_2 - c_3)\gm{1}\gm{2}\gmc{3} + (- c_1 + c_2 + c_3)\gmc{1}\gm{2}\gmc{3},\\
&\quad + (c_1 + c_2 + c_3)\gm{1}\gmc{2}\gmc{3} + (- c_1 - c_2 - c_3)\gmc{1}\gmc{2}\gmc{3}.
\end{align*}
Thus, we know from Corollary \ref{coroMultibrotIdem} that
\[c\in\cT^p(\ibasic{1}\ibasic{2},\ibasic{1}\ibasic{3},\ibasic{1}\ibasic{4})\Leftrightarrow c\in\cM^p_4 \Leftrightarrow \pm c_1 \pm c_2 \pm c_3 \in \cM^p.\]
In addition, we know from \cite{RochonRansfordParise} the intersection between a Multibrot set and the real axis:
\[\cM^p\cap\mR = \left[ -2^{1 / (p-1)}, \frac{p-1}{p^{p/(p-1)}}\right].\]
Hence, we see that
\begin{align*}
\pm c_1 \pm c_2 \pm c_3 \in \cM^p &\Leftrightarrow -2^{1 / (p-1)} \leq \pm c_1 \pm c_2 \pm c_3 \leq \frac{p-1}{p^{p/(p-1)}},\\
&\Leftrightarrow |c_1| + |c_2| + |c_3| \leq \min\left\lbrace 2^{1/(p-1)}, \frac{p-1}{p^{p/(p-1)}} \right\rbrace.
\end{align*}
Finally, we could verify that $\frac{p-1}{p^{p/(p-1)}} \leq 2^{1/(p-1)}$  for all integer $p\geq 2$, hence the result.
\end{proof}

Even though their size and center are different, since they have similar geometric shapes, we know that $\cT^p(\ibasic{1}\ibasic{2},\ibasic{1}\ibasic{3},\ibasic{1}\ibasic{4})$ and the Perplexbrot (Airbrot) are related to each other through an affine transformation. With this result, we are now ready to present the main theorem.

\begin{theorem}\label{theoTricomplexSlice}
Let $\cT_1^p(\im{m},\im{k},\im{l})$ be a principal 3D slice of $\cM_n^p$. There always exists a tricomplex principal 3D slice $\cT_2^p(\im{r},\im{q},\im{s})$ such that $\cT_1^p\sim\cT_2^p$ up to an affine transformation.
\end{theorem}
	
\begin{proof}
Let $\im{r},\im{q},\im{s}\in\mI(3)$ such that $\im{m}^2=\im{r}^2$, $\im{k}^2=\im{q}^2$ and $\im{l}^2=\im{s}^2$. Since there are at least three complex imaginary and three hyperbolic units in $\mI(3)$, it is obvious that such units exist. In addition,
\begin{enumerate}
	\item if $p$ is even and $\im{m} = 1$ or $\im{k}\im{l} = \pm \im{m}$, assume $\im{r} = 1$ or $\im{q}\im{s} = \pm\im{r}$ respectively;
	\item if $p$ is even but $\im{m} \neq 1$ and $\im{k} \im{l} \neq \pm \im{m}$, assume $\im{r} \neq 1$ and $\im{q}\im{s} \neq \pm\im{r}$;
	\item if $p$ is odd, no additional hypothesis is necessary.
\end{enumerate}

These three cases are possible in $\mI(3)$ except for the case 2 when all three units $\im{m}$, $\im{k}$ and $\im{l}$ are hyperbolic. Indeed, there are only three hyperbolic units in $\mI(3)$ and they are such that $\im{q}\im{s}=\pm\im{r}$. Thus, only in this particular case, we have that $\im{r},\im{q},\im{s}\in\mI(4)$ and, more specifically, we can consider the three units $\im{r} = \ibasic{1}\ibasic{2}$, $\im{q} = \ibasic{1}\ibasic{3}$ and $\im{s} = \ibasic{1}\ibasic{4}$.

Consider the bijective linear application $\varphi$ defined in Lemma \ref{lemEtap} which respects the previous hypotheses. Thus, we have that $\varphi(\eta^p)=\varphi(\eta)^p$ for all $\eta\in M_1$ and $\varphi\left(\mT(\im{m},\im{k},\im{l})\right) = \left(\mT(\im{r},\im{q},\im{s})\right)$. Since $\varphi$ is linear, for all $c\in\mT(\im{m},\im{k},\im{l})$, we then conclude that
\begin{gather*}
\left(\varphi\circ Q_{p,c}\right)(\eta) = \varphi\left(\eta^p + c\right) = \varphi(\eta)^p + \varphi(c) = \left(Q_{p,\varphi(c)}\circ\varphi\right)(\eta)\ \forall \eta\in M_1,\\
\Leftrightarrow \left(\varphi\circ Q_{p,c}\circ\varphi^{-1}\right)(\eta) = Q_{p,\varphi(c)}(\eta)\ \forall \eta\in M_2.
\end{gather*}
meaning that $\cT_1^p(\im{m},\im{k},\im{l})\sim\cT_2^p(\im{r},\im{q},\im{s})$ according to Definition \ref{defEquivRel}.

In most cases, this completes the proof since $\cT_2^p(\im{r},\im{q},\im{s})$ is a tricomplex principal 3D slice. However, in the particular case when $p$ is even, all three units $\im{m}$, $\im{k}$ and $\im{l}$ are hyperbolic and $\im{k}\im{l}\neq \pm\im{m}$, we have that $\im{r}$, $\im{q}$ and $\im{s}$ are quadricomplex. Nonetheless, we know that $\cT_1^p$ is equivalent to the quadricomplex slice $\cT_2^p(\ibasic{1}\ibasic{2},\ibasic{1}\ibasic{3},\ibasic{1}\ibasic{4})$ which is a regular octahedron according to Lemma \ref{lemOcta}. Since the tricomplex Perplexbrot (Airbrot) is also a regular octahedron, it is possible to apply an affine transformation on it to obtain the slice $\cT_2^p$.
\end{proof}

\section*{Conclusion}

In this article, we first presented a brief overview of the multicomplex numbers. Then, we saw that complex fractals, such as the Multibrot sets, may be extended to the multicomplex space. Since these objects have over three dimensions, it follows that, to visualize them, we have to choose three specific dimensions at a time, hence the concept of 3D slices.

It is rather interesting to explore those 3D slices. However, this exploration may seem endless since the $n$-complex space can always be generalized one step further to the $(n+1)$-complex space. Nonetheless, when considering the visualization of principal 3D slices of these fractals, Theorem \ref{theoTricomplexSlice} tells us that there is no need to generalize beyond the tricomplex space. In this context, it is optimal.

In future works, it will therefore be possible to look into Multibrot principal slices specifically in the tricomplex case. In the specific case of the Metatronbrot $\cM_3^2$, we know already from \cite{Parise} and \cite{RochonParise} that there are only eight principal 3D slices: the Tetrabrot, the Arrowheadbrot, the Hourglassbrot, the Airbrot, the Firebrot, the Mousebrot, the Metabrot and the Turtlebrot (see Figure \ref{slices}). Hence, for $p=2$, these are the only principal 3D slices of the Mandelbrot set generalized to the multicomplex space.

\begin{figure}
	\centering
	\includegraphics[width=\linewidth]{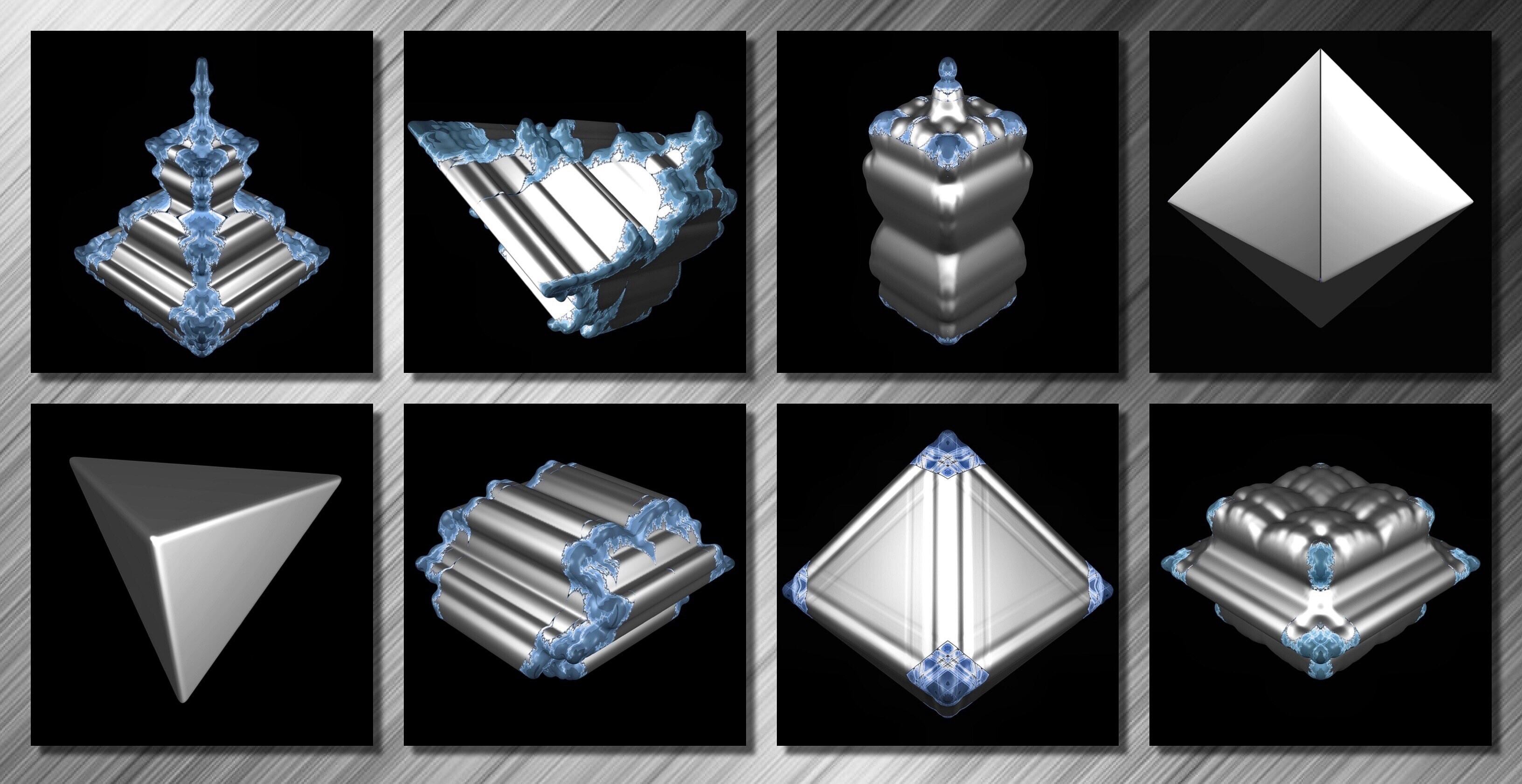}
	\caption{The eight principal 3D slices of the Metatronbrot $\cM_3^2$ representing the equivalence classes in the tricomplex case.}\label{slices}
\end{figure}

\section*{Acknowledgement}
 
DR is grateful to the Natural Sciences and Engineering Research Council of Canada (NSERC) for financial support. GB would like to thank the FRQNT and the ISM for the awards of graduate research grants. The authors are grateful to Louis Hamel and Étienne Beaulac, from UQTR, for their useful work on the MetatronBrot Explorer in Java.

\bibliographystyle{abbrv}
\bibliography{ArticleMulticomplex_Final_v04}

\end{document}